\documentclass[12pt]{article} 
\usepackage{amsfonts,amsmath,latexsym,amssymb,mathrsfs,amsthm}

\def\numberlikeadb{\global\def\theequation{\thesection.\arabic{equation}}}
\numberlikeadb
\newtheorem{theorem}{Theorem}[section]
\newtheorem{lemma}[theorem]{Lemma}

\newtheorem{definition}[theorem]{Definition}
\newtheorem{proposition}[theorem]{Proposition}

\begin{document}

\title{A Probabilistic proof of some integral formulas involving the Meijer $G$-function}
\author{Robert E. Gaunt\footnote{Department of Statistics,
University of Oxford, 24--29 St$.$ Giles, Oxford OX1 3LB, UK}
}

\date{September 2016} 
\maketitle

\begin{abstract}New integral formulas involving the Meijer $G$-function are derived using recent results concerning distributional characterisations and distributional transformations in probability theory. 
\end{abstract}


\noindent{{{\bf{AMS 2010 Subject Classification:}}} Primary 33C60; secondary 60E10 

\section{Introduction and main results}

The Meijer $G$-function is a very general function which includes many simpler special functions as special cases.  The Meijer $G$-function is defined by the contour integral:
\[G^{m,n}_{p,q}\bigg(z \; \bigg|\; {a_1,\ldots, a_p \atop b_1,\ldots,b_q} \bigg)=\frac{1}{2\pi i}\int_{c-i\infty}^{c+i\infty}z^{-s}\frac{\prod_{j=1}^m\Gamma(s+b_j)\prod_{j=1}^n\Gamma(1-a_j-s)}{\prod_{j=n+1}^p\Gamma(s+a_j)\prod_{j=m+1}^q\Gamma(1-b_j-s)}\,\mathrm{d}s,\]
where $c$ is a real constant defining a Bromwich path separating the poles of $F(s + b_j)$ from those of $F(1- a_j- s)$ and where we use the convention that the empty product is $1$.  A more detailed discussion of the Meijer $G$-function and examples are given in \cite{erdelyi}, pp. 206--222; see also \cite{olver} and references therein.  

In this paper, we derive new integral formulas involving the Meijer $G$-function.  We prove these results using a probabilistic approach, using recent results from the theory of distributional characterisations and distributional transformations in probability theory that are given \cite{gaunt ngb}.  Our main result is as follows. 

\begin{theorem}\label{thmmeijerg}Let $n$ be a positive integer and suppose that  $a_1,\ldots,a_n>-1$.  Then, for all $x>0$,
\begin{equation}\label{meigint1}G_{0,n}^{n,0}(x\,|\,a_1,\ldots,a_n)=\int_x^{\infty}G_{n,n}^{n,0}\bigg( \frac{x}{t}\; \bigg| \; \begin{matrix} a_1+1,\ldots, a_n+1 \\
a_1,\ldots,a_n \end{matrix} \bigg)G_{0,n}^{n,0}(t\,|\,a_1,\ldots,a_n)\,\mathrm{d}t.
\end{equation}
If $a_1,\ldots,a_n$ are distinct, then (\ref{meigint1}) simplifies to 
\begin{equation*}
G_{0,n}^{n,0}(x\,|\,a_1,\ldots,a_n)=\sum_{k=1}^n\bigg(\prod_{j=k}^n\frac{1}{a_j-a_k}\bigg)\int_x^{\infty}\bigg(\frac{x}{t}\bigg)^{a_k}G_{0,n}^{n,0}(t\,|\,a_1,\ldots,a_n)\,\mathrm{d}t,
\end{equation*}
and if $a_1=\cdots=a_n=a$, then (\ref{meigint1}) simplifies to 
\begin{equation}\label{lastone}G_{0,n}^{n,0}(x\,|\,a,\ldots,a)=\frac{1}{(n-1)!}\int_x^{\infty}\bigg(\frac{x}{t}\bigg)^{a}\bigg[\log\bigg(\frac{t}{x}\bigg)\bigg]^{n-1}G_{0,n}^{n,0}(t\,|\,a,\ldots,a)\,\mathrm{d}t.
\end{equation}
\end{theorem}

The result that the product of two arbitrary Meijer $G$-functions integrated over the positive real line can itself be represented as a Meijer $G$-function (the convolution theorem; see \cite{luke}, Section 5.6) is of fundamental importance.  As noted by \cite{beals}, this result lies at the heart of most comprehensive tables of integrals in print \cite{intser}.  Theorem 1.1 complements this result by giving a class of integral formulas for the product of two Meijer $G$-functions integrated over a positive half line in which the integral is itself a Meijer $G$-function.

The rest of this article is organised as follows.  In Section 2, we present some preliminary results from probability theory that we will use to prove Theorem \ref{thmmeijerg}.  In particular, we state a useful characterising equation for the product of $n$ independent gamma random variables and introduce an associated distributional transformation.  In Section 3, we establish some properties of this distributional transformation.  In Section 4,   we use these properties to prove Theorem \ref{thmmeijerg}.  We conclude by noting that the approach used in this paper to prove Theorem \ref{thmmeijerg} could in principle be used to prove other integral formulas involving special functions.

\section{Preliminary results from probability theory}

In this section, we introduce the results from probability theory that are required in our proof of Theorem \ref{thmmeijerg}.

\subsection{Products of random variables}

One of the ways the Meijer $G$-function enters probability theory is through the study of products of independent random variables.  It was shown by  \cite{springer} that probability density functions of products of independent beta, gamma and central normal random variables are Meijer $G$-functions. 
The density function of the product of $n$ independent standard normal random  variables with density $\frac{1}{\sqrt{2\pi}}\mathrm{e}^{-x^2/2}$, $x\in\mathbb{R}$, is given by
\begin{equation}\label{MeijerN} p(x)=\frac{1}{(2\pi)^{n/2}}G_{0,n}^{n,0}\bigg(\frac{x^2}{2^n} \; \bigg| \;0, \ldots,0\bigg), \quad x\in\mathbb{R}.
\end{equation}
A random variable with density (\ref{MeijerN}) has \emph{product normal} distribution with variance 1, denoted by $\mathrm{PN}(n,1)$.  The density of the product of $n$ independent gamma  random variables with density $\frac{\lambda^{r_i}}{\Gamma(r_i)}x^{r_i-1}\mathrm{e}^{-\lambda x}$, $x>0$, $\lambda>0$, $r_i>0$, $i=1,\ldots,n$, (denoted by $\mathrm{Gamma}(r_i,\lambda)$) is given by
\begin{equation}\label{MeijerGamma} p(x)= \frac{\lambda^n}{\prod_{j=1}^n\Gamma(r_j)} G_{0,n}^{n,0}( \lambda^nx \, | \, r_1-1,\ldots, r_n-1), \quad x>0,
\end{equation}
and a random variable with density (\ref{MeijerGamma}) is said to have a \emph{product gamma} distribution, which we denote by $\mathrm{PG}(r_1,\ldots,r_n,\lambda)$.  In this paper, for simplicity, we take $\lambda=1$.  Finally, the density of the product of $n$ independent beta $\mathrm{Beta}(a_i,b_i)$ random variables with density $\frac{\Gamma(a_i+b_i)}{\Gamma(a_i)\Gamma(b_i)}x^{a_i-1}(1-x)^{b_i-1}$, $0<x<1$, $a_i,b_i>0$, $i=1,\ldots,n$, (denoted by $\mathrm{Beta}(a_i,b_i)$) is given by
\begin{equation}\label{zxMeijerBC} p(x)=\bigg(\prod_{i=1}^m\frac{\Gamma(a_i+b_i)}{\Gamma(a_i)}\bigg)G_{n,n}^{n,0}\bigg(x\; \bigg| \; \begin{matrix} a_1+b_1-1,\ldots, a_n+b_n-1 \\
a_1-1,\ldots,a_n-1\end{matrix} \bigg), \quad 0<x<1.
\end{equation}

\subsection{Stein characterisations}

Recently, the products of independent beta, gamma and central normal random variables have received attention \cite{gaunt pn, gaunt ngb} 
in the context of the probabilistic technique Stein's method, introduced in 1972 by Stein \cite{stein}.  In the works \cite{gaunt pn} and \cite{gaunt ngb}, so-called Stein characterisations were obtained for products of independent beta, gamma and central normal random variables, and the Stein characterisations of products of gammas and normals will be of particular interest to us in this paper.  

Before presenting these characterising equations, we introduce some notation.  For $r\in\mathbb{R}$, we define the operator $T_r$ by $T_rf(x)=xf'(x)+rf(x)$ and we let $D$ denote the usual differential operator.  Also, let $B_{r_1,\ldots,r_n}$ denote the iterated operator $T_{r_1}\cdots T_{r_n}$.  Then we have the following characterising equations of the product normal (see \cite{gaunt pn}, Proposition 2.3) and product gamma distributions (see \cite{gaunt ngb}, Proposition 2.3):

\begin{proposition}Suppose $Z\sim\mathrm{PN}(n,\sigma^2)$.  Let $f\in C^n(\mathbb{R})$ be such that $\mathbb{E}|Zf(Z)|<\infty$ and $\mathbb{E}|Z^{k-1}f^{(k)}(Z)|<\infty$, $k=1,\ldots,n$.  Then
\begin{equation}\label{cracker} \mathbb{E}[\mathcal{A}_Zf(Z)]=0,
\end{equation}
where $\mathcal{A}_Zf(x)=DT_0^{n-1}f(x)-xf(x)$ and we set $T_0^0f(x)=f(x)$.

Suppose now that $Y\sim\mathrm{PG}(r_1,\ldots,r_n,1)$.  Let $f\in C^n(\mathbb{R}^+)$ be such that $\mathbb{E}|Yf(Y)|<\infty$ and $\mathbb{E}|Y^{k}f^{(k)}(Y)|<\infty$, $k=0,\ldots,n$, where $f^{(0)}\equiv f$.  Then
\begin{equation}\label{zxcracker11} \mathbb{E}[\mathcal{A}_Yf(Y)]=0,
\end{equation}
where $\mathcal{A}_Yf(x)=B_{r_1,\ldots,r_n}f(x)-xf(x)$.
\end{proposition}

Similar characterisations have been obtained for many standard probability distributions (see \cite{ley} for an overview of the current literature), and lie at the heart of Stein's method, by characterising distributions in a convenient manner for the purpose of deriving approximation theorems in probability theory.  For a detailed account of Stein's method for normal approximation see \cite{chen}, and for a simple, general introduction see \cite{ross}.  Whilst Stein characterisations are typically used as part of Stein's method, they have utility in other areas, such as obtaining formulas for moments of probability distributions \cite{gaunt vg} and deriving formulas for probability density functions and characteristic functions \cite{gaunt pn, gaunt ngb}.  In this paper, we shall see consider a rather curious application of the product gamma Stein characterisation (\ref{zxcracker11}) to establishing new integral formulas for the Meijer $G$-function. 

\subsection{Distributional transformations}

The characterising equation (\ref{cracker}) motivates a distributional transformation (\cite{gaunt pn}, Definition 1.2) which generalises the zero bias transformation (see \cite{goldstein}).  For $W$ a mean zero random variable with variance 1, the random variable $W^{*(n)}$ is said to have the $W$-\emph{zero biased distribution of order} $n$ if, for all $f\in C^n(\mathbb{R})$ such that the relevant expectations exist,
\begin{equation}\label{islip}\mathbb{E}[Wf(W)]=\mathbb{E}[DT_1^{n-1}f(W^{*(n)})].
\end{equation}
This distributional transformation was introduced in \cite{gaunt pn}, and a number of interesting properties were obtained, which, in conjugation with the characterisation (\ref{cracker}), allows one to prove product normal approximation theorems (see \cite{gaunt pn}, Section 4).  An analogous distributional transformation is motivated by the characterising equation (\ref{zxcracker11}):
\begin{definition}\label{def1}Let $W$ be a non-negative random variable with $0<\mathbb{E}W<\infty$.  We say that $W^{G(n)}$ has the \emph{$W$-gamma biased distribution of order $n$} with shape parameters $r_1,\ldots,r_n>0$ if, for all $f\in C^n(\mathbb{R}^+)$ such that the relevant expectations exist,
\begin{equation}\label{zxgammachar}\mathbb{E}[Wf(W)]=\mathbb{E}[B_{r_1,\ldots,r_n}f(W^{G(n)})],
\end{equation}
where $r_1,\ldots,r_n$ are such that $\prod_{k=1}^nr_k=\mathbb{E}W$.
\end{definition}
In Lemma \ref{zxetna}, we establish that for any such $W$ there exists a unique random variable $W^{G(n)}$ that has the $W$-gamma biased distribution of order $n$.  In the next section, we shall collect some useful properties of this distributional transformation.  These properties may, in future works, prove useful in deriving product gamma approximation theorems, although in this paper we will exploit its properties to prove Theorem \ref{thmmeijerg}.  

\section{Properties of the gamma bias transformation}

In this section, we establish some properties of the gamma bias transformation of order $n$ from which we shall deduce Theorem \ref{thmmeijerg}.  Firstly, we present a lemma, which gives an inverse operator for the iterated operator $B_{r_1,\ldots,r_n}=T_{r_1}\cdots T_{r_n}$.

\begin{lemma}\label{zxinvlem}Let $\hat{U}_1,\ldots,\hat{U}_n$ be independent $\mathrm{Beta}(r_j,1)$ random variables with $r_j>0$, and define $\hat{V}_n=\prod_{j=1}^n\hat{U}_j$.  Define the operator $H_{r_1,\ldots,r_n}$ by $H_{r_1,\ldots,r_n}f(x)=(\prod_{k=1}^nr_k)^{-1}\mathbb{E}f(x\hat{V}_n)$.  Then, for bounded $f:\mathbb{R}^+\rightarrow\mathbb{R}$, we have

(i) $H_{r_1,\ldots,r_n}f(x)=H_{r_1}\cdots H_{r_n}f(x)$.

(ii) $T_r H_sf(x)=f(x)+(r-s)H_sf(x)$.

(iii) $H_{r_1,\ldots,r_n}$ is the right-inverse of the operator $B_{r_1,\ldots,r_n}$ in the sense that
\begin{equation*}B_{r_1,\ldots,r_n}H_{r_1,\ldots,r_n}f(x)=f(x).
\end{equation*}

(iv) Suppose now that $f\in C^n(\mathbb{R}^+)$.  Then, for any $n\geq 1$,
\begin{equation}\label{zxleftinv}H_{r_1,\ldots,r_n}B_{r_1,\ldots,r_n}f(x)=f(x).
\end{equation}
\end{lemma}

\begin{proof}(i) We begin by obtaining a useful formula for $H_{r_1,\ldots,r_n}f(x)=(\prod_{k=1}^nr_k)^{-1}\mathbb{E}f(x\hat{V}_n)$.  We have that
\begin{equation*}H_{r_1,\ldots,r_n}f(x)=\int_{(0,1)^n}f(xu_1\cdots u_n)u_1^{r_1-1}\cdots u_n^{r_n-1}\,\mathrm{d}u_1\cdots \mathrm{d}u_n.
\end{equation*}
By a change of variables $u_n=\frac{t_n}{x}$ and $u_j=\frac{t_j}{t_{j+1}}$ for $1\leq j\leq n-1$, this can be written as
\begin{equation}\label{hforz}H_{r_1}f(x)=x^{-r_1}\int_0^xt_1^{r_1-1}f(t_1)\,\mathrm{d}t_1,
\end{equation}
and, for $n \geq 2$,
\begin{equation*}H_{r_1,\ldots,r_n}f(x)=x^{-r_n}\int_0^x\!\int_0^{t_n}\dotsi\int_0^{t_2}f(t_1)t_1^{r_1-1}t_2^{r_2-r_1-1}\cdots t_n^{r_n-r_{n-1}-1}\,\mathrm{d}t_1\mathrm{d}t_2\cdots \mathrm{d}t_n.
\end{equation*}
From these representations of $H_{r_1,\ldots,r_n}f(x)$, it is clear that $H_{r_1,\ldots,r_n}f(x)=H_{r_1}\cdots H_{r_n}f(x)$.

(ii) We now use the integral representation (\ref{hforz}) of $H_sf(x)$ to obtain
\begin{align*}T_rH_sf(x)&=x\frac{\mathrm{d}}{\mathrm{d}x}\bigg(x^{-s}\int_0^xt^{s-1}f(t)\,\mathrm{d}t\bigg)+rx^{-s}\int_0^xt^{s-1}f(t)\,\mathrm{d}t \\
&=-sx^{-s}\int_0^xt^{s-1}f(t)\,\mathrm{d}t+x^{1-s}\cdot x^{s-1}f(x)+rx^{-s}\int_0^xt^{s-1}f(t)\,\mathrm{d}t \\
&=f(x)+(r-s)H_sf(x).
\end{align*} 

(iii) From part (ii), $T_rH_rf(x)=f(x)$.  But since $B_{r_1,\ldots,r_n}f(x)=T_{r_n}\cdots T_{r_1}f(x)$ and $H_{r_1,\ldots,r_n}f(x)=H_{r_1}\cdots H_{r_n}f(x)$, it follows that $B_{r_1,\ldots,r_n}H_{r_1,\ldots,r_n}f(x)=f(x)$.

(iv) We have
\begin{align*}H_rT_rf(x)&=x^{-r}\int_0^xt^{r-1}(tf'(t)+rf(t))\,\mathrm{d}t=x^{-r}\int_0^x(t^{r}f(t))'\,\mathrm{d}t=x^{-r}\Big[t^{r}f(t)\Big]_0^x=f(x),
\end{align*}
and on using a similar argument to part (iii) it follows that $H_{r_1,\ldots,r_n}B_{r_1,\ldots,r_n}f(x)=f(x)$.
\end{proof}

We now make use of Lemma \ref{zxinvlem} to establish the existence and uniqueness of the gamma bias transformation of order $n$.  The proof of the following lemma uses a similar argument to the ones used by \cite{goldstein} and \cite{gaunt pn} to prove the existence and uniqueness of the zero bias transformation and zero bias transformation of order $n$, respectively.

\begin{lemma}\label{zxetna} Let $W$ be a non-negative random variable with $0<\mathbb{E}W<\infty$.  Then there exists a unique random variable $W^{G(n)}$ such that, for all $f\in C^n(\mathbb{R}^+)$ for which the relevant expectations exist,
\begin{equation*}\mathbb{E}Wf(W)=\mathbb{E}B_{r_1,\ldots,r_n}f(W^{G(n)}),
\end{equation*}
where $r_1,\ldots,r_n$ are positive constants such that $\prod_{k=1}^nr_k=\mathbb{E}W$.
\end{lemma}

\begin{proof}We define a linear functional $Q$ by
\[Qf=\mathbb{E}WH_{r_1,\ldots,r_n}f(W),\]
where $H_{r_1,\ldots,r_n}$ is defined as in Lemma \ref{zxinvlem}.  As $\mathbb{E}W<\infty$, it follows that $Qf$ exists for all continuous $f$ with compact support.  To see that $Q$ is positive, take $f\geq 0$.  Then $H_{r_1,\ldots,r_n}f(x)\geq0$.  Hence $\mathbb{E}WH_{r_1,\ldots,r_n}f(W)\geq 0$, and $Q$ is positive.  By the Riesz representation theorem (see, for example, \cite{folland}) we have $Qf=\int f\,\mathrm{d}\nu$, for some unique Radon measure $\nu$, which is a probability measure as $Q1=\mathbb{E}WH_{r_1,\ldots,r_n}1=\big(\prod_{k=1}^nr_k\big)^{-1}\mathbb{E}W=1$.  

We now take $f(x)=B_{r_1,\ldots,r_n}g(x),$ where $g\in C^n(\mathbb{R}^+)$, with derivatives up to $n$-th order being continuous with compact support.  Then, from (\ref{zxleftinv}), 
\[\mathbb{E}WH_{r_1,\ldots,r_n}B_{r_1,\ldots,r_n}g(W)=\mathbb{E}Wg(W),\]
which completes the proof.
\end{proof}

Combining Lemma \ref{zxetna} with the characterising equation (\ref{zxcracker11}) for the $\mathrm{PG}(r_1,\ldots,r_n,1)$ distribution immediately gives the following lemma.
\begin{lemma}\label{lemma1}Let $W$ be a non-negative random variable with $0<\mathbb{E}W<\infty$, and let $W^{G(n)}$ have the $W$-gamma biased distribution of order $n$ with shape parameters $r_1,\ldots,r_n$, in accordance with Definition \ref{def1}.  Then the $\mathrm{PG}(r_1,\ldots,r_n,1)$ distribution is the unique fixed point of the $W$-gamma biased distribution of order $n$ with shape parameters $r_1,\ldots,r_n$.
\end{lemma}

With the aid of Lemma \ref{zxinvlem}, we can obtain a useful relationship between the gamma bias distribution of order $n$ in terms of the size bias distribution, which is analogous to the relationship (see \cite{gaunt pn}) between the zero bias distribution of order $n$ in terms of the square bias distribution (defined in \cite{chen}).   If $W\geq0$ has mean $\mu>0$, we say $W^s$ has the $W$-size biased distribution if, for all $f$ such that $\mathbb{E}Wf(W)$ exists,
\[\mathbb{E}Wf(W)=\mu\mathbb{E}f(W^s).\] 
The size bias coupling is commonly used in Stein's method; for an application of this coupling to normal approximation see \cite{baldi}.

\begin{proposition}\label{zxdoublesquareg}Let $W$ be a non-negative random variable with $0<\mathbb{E}W<\infty$, and let $W^s$ have the $W$-size bias distribution.  Let $W^s$ and $\{U_k\}_{1\leq k\leq n}$ be mutually independent, with $U_k\sim \mathrm{Beta}(r_k,1)$, where $r_1,\ldots, r_n>0$ are such that $\prod_{k=1}^nr_k=\mathbb{E}W$.   Define $V_n=\prod_{k=1}^nU_k$.  Then, the random variable
\begin{equation}\label{zxsquareunif}W^{G(n)}\stackrel{\mathcal{D}}{=}V_nW^s
\end{equation}
has the $W$-gamma bias distribution of order $n$ with shape parameters $r_1,\ldots,r_n$.
\end{proposition}

\begin{proof}Let $f\in C_c$, the set of continuous functions on $\mathbb{R}^+$ with compact support.  Recall from Lemma \ref{zxinvlem} that $B_{r_1,\ldots,r_n}H_{r_1,\ldots,r_n}g(x)=g(x)$ for any $g$.  Thus,
\begin{align*}\mathbb{E}f(W^{G(n)})&=\mathbb{E}B_{r_1,\ldots,r_n}H_{r_1,\ldots,r_n}f(W^{G(n)})=\mathbb{E}WH_{r_1,\ldots,r_n}f(W)\\
&=\prod_{k=1}^n(1/r_k)\mathbb{E}Wf(V_nW)=\prod_{k=1}^n(1/r_k)\mathbb{E}W\mathbb{E}f(V_nW^s)=\mathbb{E}f(V_nW^s).
\end{align*}
Since the expectation of $f(W^{G(n)})$ and $f(V_nW^s)$ are equal for all $f\in C_c$, the random variables $W^{G(n)}$ and $V_nW^s$ must be equal in distribution.
\end{proof}

We now note some formulas for the probability density function of the product of $n$ independent beta random variables, that we will use in the proof of Proposition \ref{zxjohn hates}.

\begin{lemma}\label{zxurklem}Let $\{U_k\}_{1\leq k\leq n}$ be mutually independent $\mathrm{Beta}(r_k,1)$ random variables with $r_1,\ldots,r_n>0$.  Then, the density function of $V_n=\prod_{k=1}^nU_k$ is given by  
\begin{equation}\label{pdfprodbetass}p_{V_n}(x)=\bigg(\prod_{i=1}^nr_i\bigg)G_{n,n}^{n,0}\bigg( x\; \bigg| \; \begin{matrix} r_1,\ldots, r_n \\
r_1-1,\ldots,r_n-1 \end{matrix} \bigg), \quad 0<x<1.
\end{equation}
When the $r_k$ are distinct the density of $V_n$ can be written as  
\begin{equation}\label{zxpdfurk}p_{V_n}(x)=\bigg(\prod_{i=1}^nr_i\bigg)\sum_{k=1}^n\frac{x^{r_k-1}}{\prod_{j\not= k}^n(r_j-r_k)}, \quad 0<x<1,
\end{equation}
and the distribution function of $V_n$ is given by
\begin{equation}\label{zxcdfurk}F_{V_n}(x)=\sum_{k=1}^n\bigg(\prod_{j\not=k}^n\frac{r_j}{r_j-r_k}\bigg)x^{r_k}, \quad 0<x<1.
\end{equation}
\end{lemma}

\begin{proof}Formula (\ref{pdfprodbetass}) follows immediately from (\ref{zxMeijerBC}).  We prove that formula (\ref{zxpdfurk}) holds by induction on $n$.  The result holds for $n=1$, so suppose that for some $n\geq 1$,
\begin{equation*}p_{V_n}(v)=\bigg(\prod_{i=1}^nr_i\bigg)\sum_{k=1}^n\frac{v^{r_k-1}}{\prod_{j\not= k}^n(r_j-r_k)}, \quad 0<v<1.
\end{equation*}
By the inductive hypothesis, the joint density of $V_n$ and an independent $\mathrm{Beta}(r_{n+1},1)$ random variable $U_{n+1}$ is given by
\begin{equation*}p_{U_{n+1},V_n}(u,v)=\bigg(\prod_{i=1}^{n+1}r_i\bigg)u^{r_{n+1}-1}\sum_{k=1}^n\frac{v^{r_k-1}}{\prod_{j\not= k}^n(r_j-r_k)}, \quad 0<u,v<1.
\end{equation*}
Making the change of variables $X=U_{n+1}V_n$, we have
\begin{equation*}p_{X,V_n}(x,v)=\bigg(\prod_{i=1}^{n+1}r_i\bigg)x^{r_{n+1}-1}\sum_{k=1}^n\frac{v^{r_k-r_{n+1}-1}}{\prod_{j\not= k}^n(r_j-r_k)}, \quad 0<x<v<1,
\end{equation*}
and the marginal distribution of $X$ is given by
\begin{align*}p_X(x)&=\bigg(\prod_{i=1}^{n+1}r_i\bigg)x^{r_{n+1}-1}\sum_{k=1}^n\int_x^1\frac{v^{r_k-r_{n+1}-1}}{\prod_{j\not= k}^n(r_j-r_k)}\,\mathrm{d}v \\
&=\bigg(\prod_{i=1}^{n+1}r_i\bigg)\sum_{k=1}^n\bigg(\frac{x^{r_k-1}}{\prod_{j\not= k}^{n+1}(r_j-r_k)}-\frac{x^{r_{n+1}-1}}{\prod_{j\not=k}^{n+1}(r_j-r_k)}\bigg) \\
&=\bigg(\prod_{i=1}^{n+1}r_i\bigg)\bigg[\sum_{k=1}^n\frac{x^{r_k-1}}{\prod_{j\not= k}^{n+1}(r_j-r_k)}+\frac{x^{r_{n+1}-1}}{\prod_{j=1}^{n}(r_j-r_{n+1})}\bigg]\\
&=\bigg(\prod_{i=1}^{n+1}r_i\bigg)\sum_{k=1}^{n+1}\frac{x^{r_k-1}}{\prod_{j\not= k}^{n+1}(r_j-r_k)},
\end{align*}
which completes the inductive proof.  Formula (\ref{zxcdfurk}) for the distribution function of $V_n$ now follows immediately on integrating the formula for the density function of $V_n$ over the interval $(0,x)$.
\end{proof}

We are now in a position to prove the main result of this section: a formula for the distribution function and the density of the gamma bias transformation of order $n$.  The formula simplifies for certain values of the shape parameters $r_1,\ldots,r_n$.  

\begin{proposition}\label{zxjohn hates} Let $W$ be a random variable with $\mathbb{E}W=\prod_{k=1}^nr_k$, and let $W^{G(n)}$ have the $W$-gamma biased distribution of order $n$ with shape parameters $r_1,\ldots,r_n>0$.

(i) Let $V_n$ be distributed as the product of the mutually independent random variables $U_k\sim \mathrm{Beta}(r_k,1)$, $k=1,\ldots,n$.  Then, the distribution function of $W^{G(n)}$ is given by
\begin{equation}\label{chrbox}F_{W^{G(n)}}(w)=1-\prod_{k=1}^n(1/r_k)\mathbb{E}\bigg[W\bigg(1-F_{V_n}\bigg(\frac{w}{W}\bigg)\bigg)\mathbf{1}(W\geq w)\bigg].
\end{equation}
In particular, if the $r_k$ are all distinct, we have
\begin{equation}\label{zxdistu}F_{W^{G(n)}}(w)=\mathbb{E}\bigg[W\bigg[1-\sum_{k=1}^n\frac{1}{r_k}\bigg(\prod_{j\not=k}^n\frac{1}{r_j-r_k}\bigg)\bigg(\frac{w}{W}\bigg)^{r_k}\bigg]\mathbf{1}(W\geq w)\bigg].
\end{equation}
If $r_1=\cdots=r_n=r$ the distribution function of $W^{G(n)}$ can be written as
\begin{equation}\label{zxasdf}F_{W^{G(n)}}(w)=1-\frac{1}{(n-1)!r^n}\mathbb{E}\bigg[W\gamma\bigg(n,r\log\bigg(\frac{W}{w}\bigg)\bigg)\mathbf{1}(W\geq w)\bigg],
\end{equation}
where $\gamma(a,x)=\int_0^xt^{a-1}\mathrm{e}^{-t}\,\mathrm{d}t$.  

(ii) The density function of $W^{G(n)}$ is given by
\begin{equation}\label{chrboxp}p_{W^{G(n)}}(w)=\mathbb{E}\bigg[G_{n,n}^{n,0}\bigg( \frac{w}{W}\; \bigg| \; \begin{matrix} r_1,\ldots, r_n \\
r_1-1,\ldots,r_n-1 \end{matrix} \bigg)\mathbf{1}(W\geq w)\bigg].
\end{equation}
If the $r_k$ are distinct, we have
\begin{equation}\label{chrbox2}p_{W^{G(n)}}(w)=\mathbb{E}\bigg[\sum_{k=1}^n\bigg(\prod_{j\not=k}^n\frac{1}{r_j-r_k}\bigg)\bigg(\frac{w}{W}\bigg)^{r_k-1}\mathbf{1}(W\geq w)\bigg].
\end{equation}
If $r_1=\cdots=r_n=r$ the density function of $W^{G(n)}$ is given by
\begin{equation}\label{zxzxcv}p_{W^{G(n)}}(w)=\frac{1}{(n-1)!}\mathbb{E}\bigg[\bigg(\frac{w}{W}\bigg)^{r-1}\bigg(\log\bigg(\frac{W}{w}\bigg)\bigg)^{n-1}\mathbf{1}(W\geq w)\bigg].
\end{equation}  
\end{proposition}

\begin{proof}(i) In the proof of Proposition \ref{zxdoublesquareg} we showed that $\mathbb{E}f(W^{G(n)})=\prod_{k=1}^n(1/r_k)\mathbb{E}Wf(V_nW)$ for all bounded functions $f$.  By taking $f(x)=\mathbf{1}(x\leq w)$ we have
\begin{equation}\label{zxiiip}F_{W^{G(n)}}(w)=\prod_{k=1}^n(1/r_k)\mathbb{E}[W\mathbf{1}(V_nW\leq w)]=1-\prod_{k=1}^n(1/r_k)\mathbb{E}[W\mathbf{1}(V_nW\geq w)],
\end{equation}
as $\mathbb{E}W=\prod_{k=1}^nr_k$.   Formula (\ref{chrbox}) now follows.  If the $r_k$ are distinct, then, from formula (\ref{zxcdfurk}) for the distribution function of $V_n$ and (\ref{chrbox}), we deduce formula (\ref{zxdistu}).

Suppose now that $r_1=\cdots r_n=r$.  It is straightforward to verify that $-\log(U_k)$ follows the $\mathrm{Exp}(r)$ distribution.  Hence, $-\log(V_r)$ follows the $\mathrm{Gamma}(n,r)$ distribution, and thus
\begin{equation*}F_{W^{G(n)}}(w)=\frac{1}{r^n}\mathbb{E}\bigg[W\int_0^{-\log(\frac{w}{W})}\frac{r^n}{(n-1)!}t^{n-1}\mathrm{e}^{-rt}\,\mathrm{d}t\mathbf{1}(W\geq w)\bigg].
\end{equation*}
Making the change of variables $u=rt$ gives
\begin{equation*}\int_0^{-\log(\frac{w}{W})}t^{n-1}\mathrm{e}^{-rt}\,\mathrm{d}t=\frac{1}{r^n}\int_0^{-r\log(\frac{w}{W})}u^{n-1}\mathrm{e}^{-u}\,\mathrm{d}u=\frac{1}{r^n}\gamma\bigg(n,r\log\bigg(\frac{W}{w}\bigg)\bigg),
\end{equation*}
and formula (\ref{zxasdf}) now follows.

(ii) The general formula follows from differentiating the right-hand side of (\ref{chrbox}) with respect to $w$, and then applying formula (\ref{pdfprodbetass}) for the density of $V_n$.  Formula (\ref{chrbox2}) follows from substituting the formula (\ref{zxpdfurk}) for the density of $V_n$ into (\ref{chrboxp}).  Finally, we consider the case $r_1=\cdots=r_n=r$.  For $a>0$,  the function $\gamma(n,r\log(a/w))$ is differentiable on $(0,a)$, with derivative
\begin{equation}\label{zxqwert}\frac{\mathrm{d}}{\mathrm{d}w}\bigg[\gamma\bigg(n,r\log\bigg(\frac{a}{w}\bigg)\bigg)\bigg]=-\frac{r^n}{w}\bigg(\log\bigg(\frac{a}{w}\bigg)\bigg)^{n-1}\bigg(\frac{w}{a}\bigg)^r.
\end{equation} 
Using (\ref{zxqwert}) and dominated convergence now yields formula (\ref{zxzxcv}).
\end{proof}

\section{Proof of Theorem \ref{thmmeijerg} and concluding remarks}

\subsection{Proof of Theorem \ref{thmmeijerg}}
Let us first consider the general case $a_1,\ldots,a_n>-1$. For ease of notation, let $r_j=a_j+1$ for $j=1,\ldots,n$.  Let $W\sim\mathrm{PG}(r_1,\ldots,r_n,1)$, which has  density
\begin{equation}\label{MeijerGamma1} p(x)= K G_{0,n}^{n,0}( x \, | \, r_1-1,\ldots, r_n-1), \quad x>0,
\end{equation}
where $K=\prod_{k=1}^n(1/r_k)$.  From formula (\ref{chrboxp}), we have that the density of $W$-gamma biased distribution of order $n$ with shape parameters $r_1,\ldots,r_n$ is given by
\begin{equation}\label{gameqn33}p_{W^{G(n)}}(x)=K\int_{x}^\infty G_{n,n}^{n,0}\bigg( \frac{x}{t}\; \bigg| \; \begin{matrix} r_1,\ldots, r_n \\
r_1-1,\ldots,r_n-1 \end{matrix} \bigg)G_{0,n}^{n,0}( x \, | \, r_1-1,\ldots, r_n-1)\,\mathrm{d}t, \quad x>0.
\end{equation}
But, by Lemma \ref{lemma1}, the $\mathrm{PG}(r_1,\ldots,r_n,1)$ distribution is the unique fixed point of the $W$-gamma biased distribution of order $n$ with shape parameters $r_1,\ldots,r_n$.  Thus, (\ref{MeijerGamma1}) and (\ref{gameqn33}) are equal for all $x>0$, from which we deduce formula (\ref{meigint1}).  The formulas for the special cases of distinct $a_1,\ldots,a_n$ and $a_1=\cdots=a_n=a$ following similarly, with the difference being that we apply formulas (\ref{chrbox2}) and (\ref{zxzxcv}) instead of (\ref{chrboxp}). \hfill $\Box$

\subsection{Discussion}

The approach used in this paper to obtain the integral formulas of Theorem \ref{thmmeijerg} could also be used to arrive at integral formulas for other special functions.  The first step would be to obtain an appropriate Stein characterisation of a probability distribution $P$, whose probability density function is given in terms of special functions.  An associated distributional transformation would then have to be obtained that contains $P$ as a fixed point.  Finally, a formula for the density of the distributional transformation of $P$ would then need to be obtained, from which we would deduce an integral formula involving special functions.   

For example, the $\mathrm{PN}(n,1)$ characterisation (\ref{cracker}) and the zero bias transformation of order $n$ could be used together to obtain integral formulas involving the Meijer $G$-function.  However, doing this just leads to a formula that is equivalent to (\ref{lastone}) with $a=0$, and reduces to it after a simple change of variables.  This is essentially due to the fact that the $\Gamma(\frac{1}{2},\frac{1}{2})$ distribution, the chi-square distribution with one degree of freedom, has the same distribution as the square of a standard normal random variable.

\section*{Acknowledgements} The author is supported by EPSRC grant EP/K032402/1.  The author would like to thank the reviewer for carefully reading the manuscript and for their useful suggestions.

\end{document}